\documentclass[12pt,twoside]{amsart}
\usepackage{amssymb}
\usepackage{amscd}
\usepackage{amsmath}
\usepackage[all]{xy}

\title[Moduli b-divisors of lc-trivial fibrations]{On the moduli b-divisors of lc-trivial fibrations\\(Sur les b-diviseurs de modules des fibrations lc-triviales)
}
\author{Osamu Fujino}
\author{Yoshinori Gongyo}
\date{2014/4/9, version 1.42}
\address{Department of Mathematics, Faculty of Science,
Kyoto University, Kyoto 606-8502, Japan}
\email{fujino@math.kyoto-u.ac.jp}
\address{Graduate School of Mathematical Sciences, 
The University of Tokyo, 3-8-1 Komaba, 
Meguro, Tokyo, 153-8914 Japan.}
\email{gongyo@ms.u-tokyo.ac.jp}
\address{Department of Mathematics, Imperial College London, 180 Queen's Gate, London SW7 2AZ, UK}
\email{y.gongyo@imperial.ac.uk}

\newcommand{\mult}[0]{\operatorname{mult}}
\newcommand{\rank}[0]{\operatorname{rank}}
\newcommand{\Supp}[0]{\operatorname{Supp}}
\newcommand{\Spec}[0]{\operatorname{Spec}}
\newtheorem{thm}{Theorem}[section]

\newtheorem*{claim}{Claim}

\theoremstyle{definition}

\newtheorem{defn}[thm]{Definition}
\newtheorem{rem}[thm]{Remark}
\newtheorem*{ack}{Acknowledgments}

\newtheorem{say}[thm]{}

\begin{document}
\bibliographystyle{amsalpha+}

\maketitle

\begin{abstract}
Roughly speaking, by using the semi-stable minimal 
model program, we prove that the moduli part of an lc-trivial fibration coincides 
with that of a klt-trivial fibration induced by adjunction after taking a suitable generically finite cover. 
As an application, we obtain that the moduli part of an lc-trivial fibration 
is b-nef and abundant by Ambro's result on klt-trivial fibrations. 

\bigskip

{\noindent \textsc{R\'esum\'e}.} Grosso modo, en utilisant le programme des mod\`eles  minimaux semi-stables, nous montrons que la partie modulaire d'une fibration lc-triviale co\"incide 
avec celle d'une fibration klt-triviale induite par adjonction apr\'es changement de base  par un morphisme g\'en\'eriquement fini. 
Comme application, eu utilisant le r\'esultat de Ambro sur fibrations klt-triviales, on obtient que la partie modulaire d'une fibration lc-triviale est b-nef et abondante.
\end{abstract}

\tableofcontents
\section{Introduction}

In this paper, we prove the following theorem. 
More precisely, we reduce Theorem \ref{main} to 
Ambro's result (see \cite[Theorem 3.3]{ambro2}) by using 
the semi-stable minimal model program (see, for example, \cite{fujino3}). 
For a related result, see \cite[Theorem 1.4]{floris}. 

\begin{thm}[{cf.~\cite[Theorem 3.3]{ambro2}}]\label{main} 
Let $f:X\to Y$ be a projective surjective 
morphism between normal projective varieties with connected fibers. 
Assume that $(X, B)$ is log canonical and $K_X+B\sim _{\mathbb Q, Y}0$. Then 
the moduli $\mathbb Q$-b-divisor $\mathbf M$ is b-nef and abundant. 
\end{thm}

Let us recall the definition of {\em{b-nef and abundant}} $\mathbb Q$-b-divisors. 

\begin{defn}[{\cite[Definition 3.2]{ambro2}}]
A $\mathbb Q$-b-divisor $\mathbf M$ of a normal complete algebraic variety 
$Y$ is called {\em{b-nef and abundant}} 
if there exists a proper birational morphism $Y'\to Y$ from a normal variety $Y'$, 
endowed with a proper surjective morphism 
$h:Y'\to Z$ onto a normal variety $Z$ with connected fibers, such that: 
\begin{itemize}
\item[(1)] $\mathbf M_{Y'}\sim _\mathbb Q h^*H$, for some nef and big $\mathbb Q$-divisor $H$ of $Z$; 
\item[(2)] $\mathbf M=\overline {\mathbf M_{Y'}}$. 
\end{itemize}
\end{defn}

Let us quickly explain the idea of the proof of Theorem \ref{main}. 
We assume that the pair $(X, B)$ in Theorem \ref{main} is dlt for simplicity. 
Let $W$ be a log canonical center of $(X, B)$ which is dominant onto $Y$ and is minimal over the generic point of $Y$. 
We set $K_W+B_W=(K_X+B)|_W$ by adjunction. 
Then we have $K_W+B_W\sim _{\mathbb Q, Y}0$. 
Let $h:W\to Y'$ be the Stein factorization of $f|_W:W\to Y$. 
Note that $(W, B_W)$ is klt over the generic point of $Y'$. 
We prove that the moduli part $\mathbf M$ of $f:(X, B)\to Y$ coincides with 
the moduli part $\mathbf M^{\min}$ of $h:(W, B_W)\to Y'$ after taking a 
suitable generically finite base change by using the semi-stable minimal model program. 
We do not need the {\em{mixed}} period map nor 
the infinitesimal {\em{mixed}} Torelli theorem (see \cite[Section 2]{ambro2} and \cite{ssu}) for the proof of 
Theorem \ref{main}. 
We just reduce the problem on lc-trivial fibrations to Ambro's result on klt-trivial 
fibrations, which follows from the theory of period maps. 
Our proof of Theorem \ref{main} partially answers the questions in \cite[8.3.8 (Open problems)]{kollar}. 

It is conjectured that $\mathbf M$ is b-semi-ample (see, 
for example, \cite[0.~Introduction]{ambro1}, \cite[Conjecture 7.13.3]{ps}, 
\cite{floris}, \cite{birkar-chen}, and \cite[Section 3]{fujino10}). 
The b-semi-ampleness of the moduli part has been proved only for some special cases 
(see, for example, \cite{kawamata}, \cite{fujino1}, and \cite[Section 8]{ps}).  
See also Remark \ref{41} below. 

\begin{ack}
The first author was partially supported by the 
Grant-in-Aid for Young Scientists (A) $\sharp$24684002 
from JSPS.  
The second author was partially supported by the 
Grant-in-Aid for Research 
Activity Start-up $\sharp$24840009 from JSPS. The 
authors would like to thank Enrica Floris 
for giving them some comments 
and checking the French title and abstract. 
They also would like to thank the referee for some comments, suggestions, and 
drawing a big diagram in the proof of Theorem \ref{main}. 
\end{ack} 

We will work over $\mathbb C$, the complex number field, throughout this paper. 
We will make use of the standard notation as in \cite{fujino-funda}. 

\section{Preliminaries} 

Throughout this paper, we do not use $\mathbb R$-divisors. We only use $\mathbb Q$-divisors. 

\begin{say}[Pairs]
A pair $(X, B)$ consists of a normal variety $X$ over 
$\mathbb C$ and a $\mathbb Q$-divisor 
$B$ on $X$ such that $K_X+B$ is $\mathbb Q$-Cartier. 
A pair $(X, B)$ is called {\em{subklt}} (resp.~{\em{sublc}}) if for any projective birational morphism 
$g:Z\to X$ from a normal variety $Z$, every coefficient of $B_Z$ is $<1$ (resp.~$\leq 1$) where 
$K_Z+B_Z:=g^*(K_X+B)$. A pair $(X, B)$ is called {\em{klt}} 
(resp.~{\em{lc}}) if $(X, B)$ is subklt (resp.~sublc) and $B$ is effective. Let $(X, B)$ be an lc pair. 
If there is a 
log resolution $g:Z\to X$ of $(X, B)$ such that 
$\mathrm{Exc}(g)$ is a divisor and that 
the coefficients of the $g$-exceptional part of $B_Z$ 
are $<1$, then 
the pair $(X, B)$ is called 
{\em{divisorial log terminal}} ({\em{dlt}}, for short). 
Let $(X, B)$ be a sublc pair and let $W$ be a closed subset of $X$. Then $W$ is called a {\em{log canonical 
center}} of $(X, B)$ if 
there are a projective birational morphism $g:Z\to X$ from a normal 
variety $Z$ and a prime divisor $E$ on $Z$ such that $\mult _EB_Z=1$ and 
that $g(E)=W$. Moreover we say that $W$ is {\em{minimal}} if it is minimal with 
respect to inclusion. 
\end{say}

In this paper, we use the notion of {\em{b-divisors}} introduced by Shokurov. 
For details, we refer to \cite[2.3.2]{corti} and \cite[Section 3]{fujino4}. 

\begin{say}
[Canonical b-divisors] Let $X$ be a normal variety and let $\omega$ be a top rational differential 
form of $X$. 
Then $(\omega)$ defines a b-divisor $\mathbf K$. 
We call $\mathbf K$ the {\em{canonical b-divisor}} of $X$. 
\end{say}

\begin{say}[$\mathbf A(X, B)$ and $\mathbf A^*(X, B)$] 
The {\em{discrepancy b-divisor}} $\mathbf A=\mathbf A(X, B)$ of a pair $(X, B)$ is the $\mathbb Q$-b-divisor 
of $X$ with the trace $\mathbf A_Y$ defined by the 
formula 
$$
K_Y=f^*(K_X+B)+\mathbf A_Y, 
$$ 
where $f:Y\to X$ is a proper birational morphism of normal varieties. 
Similarly, we define $\mathbf A^*=\mathbf A^*(X, B)$ by 
$$
\mathbf A_Y^*=\sum _{a_i>-1}a_i A_i
$$ 
for 
$$
K_Y=f^*(K_X+B)+\sum a_i A_i, 
$$ 
where $f:Y\to X$ is a proper birational morphism of normal varieties. 
Note that $\mathbf A(X, B)=\mathbf A^*(X, B)$ when $(X, B)$ is subklt. 

By the definition, we have 
$\mathcal O_X(\lceil \mathbf A^*(X, B)\rceil)\simeq \mathcal O_X$ 
if $(X, B)$ is lc (see \cite[Lemma 3.19]{fujino4}). We also 
have $\mathcal O_X(\lceil \mathbf A(X, B)\rceil)
\simeq \mathcal O_X$ when $(X, B)$ is klt. 
\end{say}

\begin{say}[b-nef and b-semi-ample $\mathbb Q$-b-divisors] 
Let $X$ be a normal variety and let $X\to S$ be a proper surjective morphism onto a variety $S$. 
A $\mathbb Q$-b-divisor $\mathbf D$ of $X$ is {\em{b-nef over $S$}} 
(resp.~{\em{b-semi-ample over $S$}}) if there exists a proper birational morphism $X'\to X$ from a normal 
variety $X'$ such that $\mathbf D=\overline {\mathbf D_{X'}}$ and $\mathbf D_{X'}$ is nef 
(resp.~semi-ample) relative to the induced morphism $X'\to S$. 
\end{say} 

\begin{say}
Let $D=\sum _i d_iD_i$ be a $\mathbb Q$-divisor on a normal variety, where 
$D_i$ is a prime divisor for every $i$, $D_i\ne D_j$ for $i\ne j$, and 
$d_i\in \mathbb Q$ for every $i$. Then we 
set 
$$
D^{\geq 0}=\sum _{d_i\geq 0}d_iD_i \quad \text{and} \quad D^{\leq 0}=\sum _{d_i\leq 0}d_iD_i. 
$$
\end{say}

\section{A quick review of lc-trivial fibrations}

In this section, we quickly recall some basic definitions and results on 
{\em{klt-trivial fibrations}} and {\em{lc-trivial fibrations}} (see also \cite[Section 3]{fujino10}). 

\begin{defn}[Klt-trivial fibrations]\label{def-klt}
A {\em{klt-trivial fibration}} $f:(X, B)\to Y$ consists of a proper 
surjective morphism $f:X\to Y$ between normal varieties with connected fibers and 
a pair $(X, B)$ satisfying the following properties: 
\begin{itemize}
\item[(1)] $(X, B)$ is subklt over the generic point of $Y$; 
\item[(2)] $\rank f_*\mathcal O_X(\lceil \mathbf A(X, B)\rceil)=1$; 
\item[(3)] There exists a $\mathbb Q$-Cartier $\mathbb Q$-divisor 
$D$ on $Y$ such that 
$$
K_X+B\sim _{\mathbb Q}f^*D. 
$$
\end{itemize} 
\end{defn}

Note that Definition \ref{def-klt} is nothing but \cite[Definition 2.1]{ambro1}, 
where a klt-trivial fibration is called an lc-trivial fibration. So, our 
definition of lc-trivial fibrations in Definition \ref{def-lc} is different from the 
original one in \cite[Definition 2.1]{ambro1}. 

\begin{defn}[Lc-trivial fibrations]\label{def-lc}
An {\em{lc-trivial fibration}} $f:(X, B)\to Y$ consists of a proper 
surjective morphism $f:X\to Y$ between normal varieties with connected fibers and 
a pair $(X, B)$ satisfying the following properties: 
\begin{itemize}
\item[(1)] $(X, B)$ is sublc over the generic point of $Y$; 
\item[(2)] $\rank f_*\mathcal O_X(\lceil \mathbf A^*(X, B)\rceil)=1$; 
\item[(3)] There exists a $\mathbb Q$-Cartier $\mathbb Q$-divisor 
$D$ on $Y$ such that 
$$
K_X+B\sim _{\mathbb Q}f^*D. 
$$
\end{itemize} 
\end{defn}

In Section \ref{sec4}, we sometimes take various base changes and 
construct the induced lc-trivial fibrations and klt-trivial fibrations. For the 
details, see \cite[Section 2]{ambro1}. 

\begin{say}[Induced lc-trivial fibrations by base changes]\label{33}
Let $f:(X, B)\to Y$ be a klt-trivial (resp.~an lc-tirivial) fibration and let $\sigma:Y'\to Y$ be a generically finite 
morphism. Then we have an induced klt-trivial (resp.~lc-trivial) fibration $f':(X', B_{X'})\to Y'$, where 
$B_{X'}$ is defined by $\mu^*(K_X+B)=K_{X'}+B_{X'}$: 
$$
\xymatrix{
   (X', B_{X'}) \ar[r]^{\mu} \ar[d]_{f'} & (X, B)\ar[d]^{f} \\
   Y' \ar[r]_{\sigma} & Y,
} 
$$
Note that $X'$ is the normalization of 
the main component of $X\times _{Y}Y'$. 
We sometimes replace $X'$ with $X''$ where $X''$ is a normal variety 
such that there is a proper birational morphism 
$\varphi:X''\to X'$. 
In this case, we set 
$K_{X''}+B_{X''}=\varphi^*(K_{X'}+B_{X'})$.  
\end{say}

Let us explain the definitions of the {\em{discriminant}} and {\em{moduli}} $\mathbb Q$-b-divisors. 

\begin{say}
[Discriminant $\mathbb Q$-b-divisors and moduli $\mathbb Q$-b-divisors] 
Let $f:(X, B)\to Y$ be an lc-trivial fibration as in Definition \ref{def-lc}. 
Let $P$ be a prime divisor on $Y$. 
By shrinking $Y$ around the generic point of $P$, 
we assume that $P$ is Cartier. We set 
$$
b_P=\max \left\{t \in \mathbb Q\, \left|\, 
\begin{array}{l}  {\text{$(X, B+tf^*P)$ is sublc over}}\\
{\text{the generic point of $P$}} 
\end{array}\right. \right\} 
$$ 
and 
set $$
B_Y=\sum _P (1-b_P)P, 
$$ 
where $P$ runs over prime divisors on $Y$. Then it is easy to  see that 
$B_Y$ is a well-defined $\mathbb Q$-divisor on $Y$ and is called the {\em{discriminant 
$\mathbb Q$-divisor}} of $f:(X, B)\to Y$. We set 
$$
M_Y=D-K_Y-B_Y
$$ 
and call $M_Y$ the {\em{moduli $\mathbb Q$-divisor}} of $f:(X, B)\to Y$. 
Let $\sigma:Y'\to Y$ be a proper birational morphism 
from a normal variety $Y'$ and let $f':(X', B_{X'})\to Y'$ be the induced lc-trivial fibration 
by $\sigma:Y'\to Y$ (see \ref{33}). We can define $B_{Y'}$, $K_{Y'}$ and $M_{Y'}$ such that 
$\sigma^*D=K_{Y'}+B_{Y'}+M_{Y'}$, 
$\sigma_*B_{Y'}=B_Y$, $\sigma _*K_{Y'}=K_Y$ and $\sigma_*M_{Y'}=M_Y$. Hence 
there exist a unique $\mathbb Q$-b-divisor $\mathbf B$ such that 
$\mathbf B_{Y'}=B_{Y'}$ for every $\sigma:Y'\to Y$ and a unique 
$\mathbb Q$-b-divisor $\mathbf M$ such that $\mathbf M_{Y'}=M_{Y'}$ for 
every $\sigma:Y'\to Y$. 
Note that $\mathbf B$ is called the {\em{discriminant $\mathbb Q$-b-divisor}} and 
that $\mathbf M$ is called the {\em{moduli $\mathbb Q$-b-divisor}} associated to $f:(X, B)\to Y$. 
We sometimes simply say that $\mathbf M$ is the {\em{moduli part}} of $f:(X, B)\to Y$. 
\end{say}

For the basic properties of 
the discriminant and moduli $\mathbb Q$-b-divisors, see \cite[Section 2]{ambro1}. 

Let us recall the main theorem of \cite{ambro1}. Note that 
a klt-trivial fibration in the sense of Definition \ref{def-klt} is called 
an lc-trivial fibration in \cite{ambro1}. 

\begin{thm}[{see \cite[Theorem 2.7]{ambro1}}]\label{thm-klt-tri}
Let $f:(X, B)\to Y$ be a klt-trivial fibration and let $\pi:Y\to S$ be a proper morphism. 
Let $\mathbf B$ and $\mathbf M$ be the induced discriminant and moduli $\mathbb Q$-b-divisors of $f$. 
Then, 
\begin{itemize}
\item[(1)] $\mathbf K+\mathbf B$ is $\mathbb Q$-b-Cartier, 
that is, there exists a proper birational morphism $Y'\to Y$ from a normal 
variety $Y'$ such that $\mathbf {K}+\mathbf {B}=\overline{K_{Y'}+\mathbf{B}_{Y'}}$, 
\item[(2)] $\mathbf M$ is b-nef over $S$.  
\end{itemize}
\end{thm} 

Theorem \ref{thm-klt-tri} has some important applications, see, for example, 
\cite[Proof of Theorem 1.1]{fujino-kawa} and 
\cite[The proof of Theorem 1.1]{fujino4}. 

By modifying the arguments in \cite[Section 5]{ambro1} suitably with the 
aid of \cite[Theorems 3.1, 3.4, and 3.9]{fujino2} (see also \cite{fujino-fujisawa}), 
we can generalize Theorem \ref{thm-klt-tri} 
as follows. 

\begin{thm}\label{thm-lc-tri} 
Let $f:(X, B)\to Y$ be an lc-trivial fibration and let $\pi:Y\to S$ be a proper morphism. 
Let $\mathbf B$ and $\mathbf M$ be the induced discriminant and moduli $\mathbb Q$-b-divisors of $f$. 
Then, 
\begin{itemize}
\item[(1)] $\mathbf K+\mathbf B$ is $\mathbb Q$-b-Cartier, 
\item[(2)] $\mathbf M$ is b-nef over $S$.  
\end{itemize}
\end{thm} 

Theorem \ref{thm-klt-tri} is proved by using the theory of variations of 
Hodge structure. On the 
other hand, Theorem \ref{thm-lc-tri} follows from the theory of 
variations of {\em{mixed}} Hodge structure. 
We do not adopt the formulation in \cite[Section 4]{fujino-pre} (see also \cite[8.5]{kollar}) 
because the argument in \cite{ambro1} suits 
our purposes better. 
For the reader's convenience, we include the main ingredient of the proof of Theorem \ref{thm-lc-tri}, which 
easily follows from \cite[Theorems 3.1, 3.4, and 3.9]{fujino2} (see also \cite{fujino-fujisawa}). 

\begin{thm}[{cf.~\cite[Theorem 4.4]{ambro1}}]\label{thm-mhs} 
Let $f:X\to Y$ be a projective morphism between algebraic varieties. 
Let $\Sigma_X$ {\em{(}}resp.~$\Sigma_Y${\em{)}} be 
a simple normal crossing divisor on $X$ {\em{(}}resp.~$Y${\em{)}} such that 
$f$ is smooth over $Y\setminus \Sigma_Y$, $\Sigma_X$ is relatively normal crossing 
over $Y\setminus \Sigma_Y$, and $f^{-1}(\Sigma_Y)\subset \Sigma _X$. 
Assume that $f$ is semi-stable in codimension one. 
Let $D$ be a simple normal crossing divisor on $X$ such that 
$\Supp D\subset \Sigma_X$ and 
that every irreducible component of $D$ is dominant onto $Y$. Then the following properties hold. 
\begin{itemize}
\item[(1)] $R^pf_*\omega_{X/Y}(D)$ is a locally free sheaf on $Y$ for every $p$. 
\item[(2)] $R^pf_*\omega_{X/Y}(D)$ is semi-positive for every $p$. 
\item[(3)] Let $\rho :Y'\to Y$ be a projective morphism from a smooth variety $Y'$ such that 
$\Sigma_{Y'}=\rho^{-1}(\Sigma_Y)$ is a simple normal crossing 
divisor on $Y'$. 
Let $\pi:X'\to X\times _YY'$ be a resolution of the main component of $X\times _YY'$ such that 
$\pi$ is an isomorphism over $Y'\setminus \Sigma_{Y'}$. Then 
we obtain the following commutative diagram{\em{:}} 
$$
\xymatrix{
   X' \ar[r] \ar[d]_{f'} & X\ar[d]^{f} \\
   Y' \ar[r]_{\rho} & Y. 
}  
$$ 
Assume that $f'$ is projective, $D'$ is a simple normal crossing divisor on $X'$ such that 
$D'$ coincides with 
$D\times _YY'$ over $Y'\setminus \Sigma_{Y'}$, and every stratum of $D'$ is dominant onto 
$Y'$. 
Then there exists a natural isomorphism 
$$
\rho^*(R^pf_*\omega_{X/Y}(D))\simeq R^pf'_*\omega_{X'/Y'}(D')
$$ 
which extends the base change isomorphism over $Y\setminus \Sigma_Y$ for 
every $p$. 
\end{itemize}
\end{thm}

\begin{rem} 
For the proof of Theorem \ref{thm-lc-tri}, 
Theorem \ref{thm-mhs} for $p=0$ is sufficient. 
Note that all the local monodromies on 
$R^q(f_{0})_*\mathbb C_{X_0\setminus D_0}$ around $\Sigma_Y$ are unipotent
for every $q$ because $f$ is semi-stable in codimension one, where 
$X_0=f^{-1}(Y\setminus \Sigma_Y)$, $D_0=D|_{X_0}$, and $f_0=f|_{X_0\setminus D_0}$. 
More precisely, let $C_0^{[d]}$ be 
the disjoint union of all the 
codimension $d$ log canonical centers of $(X_0, D_0)$. 
If $d=0$, then we put $C_0^{[0]}=X_0$. 
In this case, we have the following weight spectral sequence 
$$
_W\!E_1^{-d, q+d}=R^{q-d}(f|_{C_0^{[d]}})_*
\mathbb C_{C_0^{[d]}}\Longrightarrow R^q(f_0)_*\mathbb C_{X_0\setminus D_0}
$$ 
which degenerates at $E_2$ (see, for example, \cite[Corollaire (3.2.13)]{deligne}). 
Since $f$ is semi-stable in codimension one, all the local monodromies 
on $R^{q-d}(f|_{C_0^{[d]}})_*\mathbb C_{C_0^{[d]}}$ around $\Sigma_Y$ are 
unipotent for every $q$ and $d$ 
(see, for example, \cite[VII.~The Monodromy theorem]{katz}). 
By the above spectral sequence, we obtain that 
all the local monodromies on $R^q(f_0)_*\mathbb C_{X_0\setminus D_0}$ around 
$\Sigma_Y$ are unipotent. 
\end{rem}

We add a remark on the proof of Theorem \ref{thm-lc-tri}. 
In Remark \ref{rem39}, we explain how to modify the arguments in the proof of \cite[Lemma 5.2]{ambro1} 
in order to treat lc-trivial fibrations. 
It will help the reader to understand the main difference between klt-trivial fibrations and lc-trivial fibrations and the 
reason why we need Theorem \ref{thm-mhs}. 

\begin{rem}\label{rem39} 
We use the notation in \cite[Lemma 5.2]{ambro1}. We only assume that 
$(X, B)$ is sublc over the generic point of $Y$ in \cite[Lemma 5.2]{ambro1}. 
We write 
$$
B=\sum _{i\in I}d_iB_i
$$ 
where $B_i$ is a prime divisor for every $i$ and $B_i\ne B_j$ for $i\ne j$. 
We set 
$$
J=\left\{i\in I \, \left|\,  {\text{$B_i$ is dominant onto $Y$ 
and $d_i=1$}} \right. \right\} 
$$ 
and set 
$$
D=\sum _{i\in J}B_i. 
$$
In Ambro's original setting in \cite[Lemma 5.2]{ambro1}, we have 
$D=0$ because 
$(X, B)$ is subklt over the generic point of $Y$. 
In the proof of \cite[Lemma 5.2 (4)]{ambro1}, 
we have to replace 
$$
\widetilde f_*\omega_{\widetilde {X}/Y}=
\bigoplus _{i=0}^{b-1}f_*\mathcal O_X(\lceil (1-i)K_{X/Y}-iB+if^*B_Y+if^*M_Y\rceil)\cdot \psi^i. 
$$ 
with  
$$
\widetilde f_*\omega_{\widetilde {X}/Y}(\pi^*D)=
\bigoplus _{i=0}^{b-1}f_*\mathcal O_X(\lceil (1-i)K_{X/Y}-iB+D+if^*B_Y+if^*M_Y\rceil)\cdot \psi^i 
$$ 
in order to treat lc-trivial fibrations. 
We leave the details as exercises for the reader. 
\end{rem}

The following theorem is a part of \cite[Theorem 3.3]{ambro2}. 

\begin{thm}[{see \cite[Theorem 3.3]{ambro2}}]\label{thm-moduli} 
Let $f:(X, B)\to Y$ be a klt-trivial fibration such that 
$Y$ is complete,  
the geometric generic fiber $X_{\overline \eta}=X\times \Spec \overline {\mathbb C(\eta)}$ is a projective variety, 
and 
$B_{\overline \eta}=B|_{X_{\overline {\eta}}}$ is effective, where 
$\eta$ is the generic point of $Y$. 
Then the moduli $\mathbb Q$-b-divisor $\mathbf M$ is b-nef and abundant. 
\end{thm}

\section{Proof of Theorem \ref{main}}\label{sec4} 

Let us give a proof of Theorem \ref{main}. 

\begin{proof}[Proof of Theorem \ref{main}] 
By taking a dlt blow-up, we may assume that 
the pair $(X, B)$ is $\mathbb Q$-factorial and dlt (see, for example, \cite[Section 4]{fujino3}). 
If $(X, B)$ is klt over the generic point of $Y$, then Theorem \ref{main} follows from \cite[Theorem 3.3]{ambro2} 
(see Theorem \ref{thm-moduli}). 
Therefore, we may also assume that $(X, B)$ is not klt over the generic point of $Y$. 
Let $\sigma_1:Y_1\to Y$ be a suitable 
projective birational morphism such that $\mathbf M=\overline {\mathbf M_{Y_1}}$ 
and $\mathbf M_{Y_1}$ is nef by Theorem \ref{thm-lc-tri}. 
Let $W$ be an arbitrary log canonical center of $(X, B)$ which is dominant onto $Y$ and 
is minimal over the 
generic point of $Y$. We set 
$$K_W+B_W=(K_X+B)|_W$$ 
by adjunction (see, for example, \cite[3.9]{fujino-what}). By the construction, we have $K_W+B_W
\sim _{\mathbb Q, Y}0$. 
We consider the Stein factorization of $f|_W:W\to Y$ and denote it 
by $h:W\to Y'$. Then $K_W+B_W\sim _{\mathbb Q, Y'}0$. 
We see that $h:(W, B_W)\to Y'$ is a klt-trivial fibration since the general fibers of $f|_{W}$ are 
klt pairs. 
Let $Y_2$ be a suitable resolution of $Y'$ which factors through $\sigma_1:Y_1\to Y$. 
By taking the base change by $\sigma_2:Y_2\to Y_1$, 
we obtain $\mathbf M_{Y_2}=\sigma_2^*\mathbf M_{Y_1}$ 
(see \cite[Proposition 5.5]{ambro1}). 
Note that the proof of \cite[Proposition 5.5]{ambro1} works for 
lc-trivial fibrations by some suitable modifications. 
By the construction, on the induced lc-trivial fibration $f_2:(X_2, B_{X_2})\to Y_2$, 
where $X_2$ is the normalization of the main component of 
$X\times _YY_2$, 
there is a log canonical 
center $W_2$ of $(X_2, B_{X_2})$ such that $f_2|_{W_2^\nu}: (W_2^\nu, B_{W_2^\nu})\to Y_2$ is a 
klt-trivial fibration, 
which is birationally equivalent to $h:(W, B_W)\to Y'$. 
Note that $\nu: W_2^\nu\to W_2$ 
is the normalization, $K_{W_2^\nu}+B_{W_2^\nu}=\nu^*(K_{X_2}+B_{X_2})|_{W_2}$, 
and $f_2|_{W_2^\nu}=f_2|_{W_2}\circ \nu$. 
It is easy to see that  
$$
K_{Y_2}+\mathbf M_{Y_2}+\mathbf B_{Y_2}\sim _{\mathbb Q}
K_{Y_2}+\mathbf M^{\min}_{Y_2}+\mathbf B^{\min}_{Y_2}
$$ 
where $\mathbf M^{\min}$ and $\mathbf B^{\min}$ are the induced moduli 
and discriminant $\mathbb Q$-b-divisors 
of $f_2|_{W^\nu_2}: (W^\nu_2, B_{W^\nu_2})\to Y_2$ such that 
$$
K_{W^\nu_2}+B_{W^\nu_2}\sim _{\mathbb Q}(f_2|_{W^\nu_2})
^*(K_{Y_2}+\mathbf M_{Y_2}^{\min}+\mathbf B_{Y_2}^{\min}). 
$$ 
By replacing $Y_2$ birationally, we may further assume that 
$\mathbf M^{\min}=\overline {\mathbf M_{Y_2}^{\min}}$ by Theorem \ref{thm-klt-tri}. 
By Theorem \ref{thm-moduli}, we see that $\mathbf M_{Y_2}^{\min}$ is nef and abundant. 
Let $\sigma_3:Y_3\to Y_2$ be a suitable 
generically finite morphism such that  
the induced lc-trivial fibration $f_3:(X_3, B_{X_3})\to Y_3$ has a semi-stable resolution in codimension one 
(see, for example, \cite{kkms}, \cite[(9.1) Theorem]{ssu}, and \cite[Theorem 4.3]{ambro1}). 
Note that $X_3$ is the normalization of the main component of $X\times _YY_3$. 
Here we draw the following big diagram for the reader's convenience. 
$$
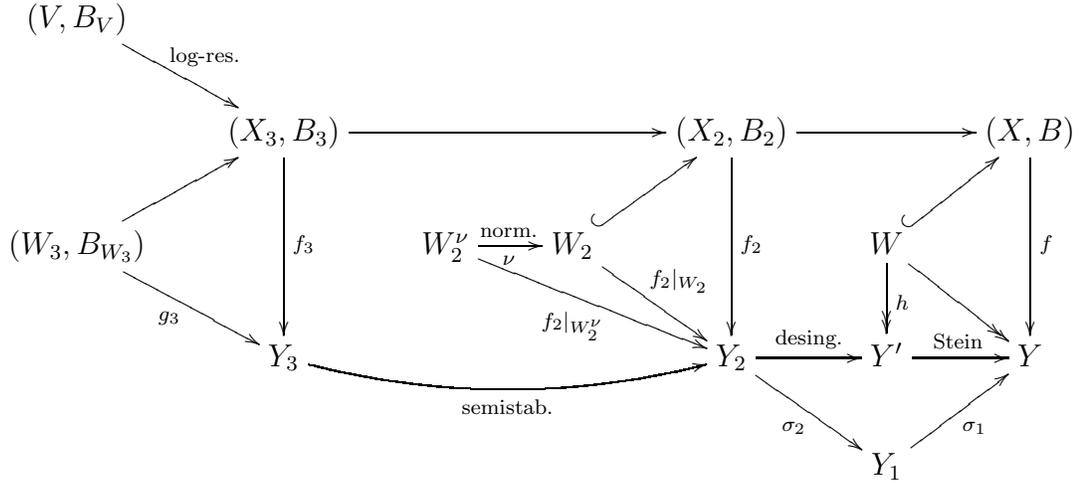
\xymatrix{
(V, B_V) \ar[dr]^{\textrm{log-res.}} & & & & & &\\
& (X_3,B_3) \ar[rrr] \ar[dd]^{f_3}
& & & (X_2,B_2)\ar[rr]\ar[dd]^{f_2}
& & (X, B)\ar[dd]^{f}
\\(W_3,B_{W_3})
\ar[dr]_{g_3}
\ar[ur]
& & W_2^\nu
\ar[drr]_{{f_2}|_{W_2^\nu}}
\ar[r]_\nu^{\textrm {norm.}}
& W_2 
\ar@{^{(}->}[ur]
\ar[dr]^{{f_2}|_{W_2}}
& & W
\ar@{^{(}->}[ur] 
\ar@{->>}[d]^h\ar@{->>}[dr] 
&\\
&Y_3
\ar@/_1pc/[rrr]_{\textrm {semistab.}}
&&&
Y_2\ar[r]^{\textrm {desing.}}\ar[dr]_{\sigma_2}
&
Y'\ar[r]^{\textrm {Stein}}
& 
Y
\\
 & & & & & Y_1\ar[ur]_{\sigma_1} &
}
$$
Note that $g_3:(W_3, B_{W_3})\to 
Y_3$ is the induced klt-trivial 
fibration from $f_2|_{W_2^\nu}: W_2^\nu\to Y_2$ by $\sigma_3:Y_3\to 
Y_2$. 
On $Y_3$, we will see the following claim by using the semi-stable minimal model 
program.
\begin{claim} The following equality 
$$
\mathbf B_{Y_3}=\mathbf B_{Y_3}^{\min}
$$ 
holds. 
\end{claim} 
\begin{proof}[Proof of Claim] 
By taking general hyperplane cuts, we may assume that $Y_3$ is a curve. 
We write 
$$
\mathbf B_{Y_3}=\sum _P(1-b_P)P \quad {\text{and}}\quad \mathbf B_{Y_3}^{\min}=\sum _P(1-b_P^{\min})P. 
$$
Let $\varphi: (V, B_V)\to (X_3, B_{X_3})$ be a resolution of $(X_3, B_{X_3})$ with the following properties: 
\begin{itemize}
\item $K_V+B_V=\varphi^*(K_{X_3}+B_{X_3})$; 
\item $\pi^*Q$ is a reduced simple normal crossing divisor on $V$ for every $Q\in Y_3$, where 
$\pi: V\to X_3\to Y_3$; 
\item $\Supp \pi^*Q\cup \Supp B_V$ is a simple normal crossing divisor on $V$ for every $Q\in Y_3$; 
\item $\pi$ is projective.  
\end{itemize}
Let $\Sigma$ be a reduced divisor on $Y_3$ such that $\pi$ is smooth over $Y_3\setminus \Sigma$ and 
that $\Supp 
B_V$ is relatively normal crossing over $Y_3\setminus \Sigma$. 
We consider the set of prime divisors $\{E_i\}$ where $E_i$ is a prime divisor on $V$ such that 
$\pi(E_i)\in \Sigma$ and 
$$\mult _{E_i} (B_V+\sum _{P\in\Sigma}b_P\pi^*P)^{\geq 0}<1. $$ 
We run the minimal model programs with ample scaling with respect to 
$$
K_V+(B_V+\sum _{P\in \Sigma}b_P\pi^*P)^{\geq 0}+\varepsilon \sum_i E_i 
$$ 
over $X_3$ and $Y_3$ for some small positive rational number $\varepsilon$. 
Note that 
$$
(V, (B_V+\sum _P b_P\pi^*P)^{\geq 0}+\varepsilon \sum _i E_i)
$$ 
is a $\mathbb Q$-factorial dlt pair because $0<\varepsilon \ll 1$. 
We set $$E=-(B_V+\sum _P b_P\pi^*P)^{\leq0}+\varepsilon \sum_i E_i. $$
Then it holds that   
$$
K_V+(B_V+\sum _Pb_P\pi^*P)^{\geq 0}+\varepsilon \sum_i E_i \sim _{\mathbb Q, Y_3}E\geq 0. 
$$ 
First we run a minimal model program with ample scaling with respect to 
$$
K_V+(B_V+\sum _Pb_P\pi^*P)^{\geq 0}+\varepsilon \sum_i E_i \sim _{\mathbb Q, X_3}E\geq 0
$$
over $X_3$. Note that every irreducible component of 
$E$ which is dominant onto $Y_3$ is exceptional 
over $X_3$ by the construction. 
Thus, if $E$ is dominant onto $Y_3$, then it is not 
contained in the relative movable cone over $X_3$. Therefore, 
after finitely many steps, we may assume that every irreducible 
component of $E$ is contained in a fiber over $Y_3$ (see, 
for example, \cite[Theorem 2.2]{fujino3}). Next we run a minimal model 
program with ample scaling with respect to 
$$
K_V+(B_V+\sum _Pb_P\pi^*P)^{\geq 0}+\varepsilon \sum_i E_i \sim _{\mathbb Q, Y_3}E\geq 0
$$ 
over $Y_3$. 
Then the minimal model program terminates at $V'$ (see, for example, \cite[Theorem 2.2]{fujino3}). 
Note that all the components of $E+\sum _i E_i$ are contracted by 
the above minimal model programs. 
Thus, we have 
$$
K_{V'}+(B_{V'}+\sum _P b_P{\pi'}^*P)^{\geq 0}\sim _{\mathbb Q, Y_3}0, 
$$
where $\pi':V'\to Y_3$ and $B_{V'}$ is the pushforward of $B_V$ on $V'$. 
Note that $B_{V'}+\sum _Pb_P\pi'^*P$ is effective 
since $\Supp(E+\sum _i E_i)$ is contracted by the above minimal model programs. 
Of course, we see that 
$$
(V', (B_{V'}+\sum _P b_P\pi'^*P)^{\geq 0})=(V', B_{V'}+\sum _P b_P\pi'^*P)
$$ 
is a $\mathbb Q$-factorial dlt pair. 
By the construction, the induced proper birational map 
$$
(V, B_V+\sum _Pb_P\pi^*P)\dashrightarrow (V', B_{V'}+\sum _P b_P \pi'^*P)
$$ 
over $Y_3$ is $B$-birational (see \cite[Definition 1.5]{fujino-abun}), that is, 
we have a common resolution 
\begin{equation*}
\xymatrix{ & Z\ar[dl]_{a} \ar[dr]^{b}\\
 V \ar@{-->}[rr]  & & V'}
\end{equation*} 
over $Y_3$ such that 
$$
a^*(K_V+B_V+\sum _{P\in \Sigma}b_P\pi^*P)=b^*(K_{V'}+B_{V'}+\sum _{P\in \Sigma}b_P\pi'^*P). 
$$
Let $S$ be any log canonical center of $(V', B_{V'}+\sum _Pb_P{\pi'}^*P)$ which is 
dominant onto $Y_3$ and is minimal over the generic point of $Y_3$. 
Then $(S, B_{S})$, where 
$$
K_{S}+B_{S}=(K_{V'}+B_{V'}+\sum _Pb_P{\pi'}^*P)|_{S}, 
$$ 
is not klt but lc over every $P\in \Sigma$ since it holds that 
\begin{align}\label{shiki1} 
B_{V'}+\sum _{P\in \Sigma}b_P\pi'^*P\geq \sum _{P\in \Sigma}\pi'^*P. \tag{$\spadesuit$}
\end{align}
Note that \eqref{shiki1} follows from the fact that all the components of $\sum _i E_i$ are contracted 
in the minimal model 
programs. 
Let $g_3:(W_3, B_{W_3})\to Y_3$ be the induced klt-trivial fibration from $(W^\nu_2, 
B_{W^\nu_2})\to Y_2$ by $\sigma_2:Y_3\to Y_2$. 
By \cite[Claims $(A_n)$ and $(B_n)$ in the proof of Lemma 4.9]{fujino-abun},  there is a log canonical center 
$S_0$ of $(V', B_{V'}+\sum _Pb_P\pi'^*P)$ which is dominant onto $Y_3$ and is minimal 
over the generic point of $Y_3$ such that 
there is a $B$-birational map 
$$
(W_3, B_{W_3}+\sum _{P\in\Sigma}b_Pg_3^*P)\dashrightarrow (S_0, B_{S_0})
$$ 
over $Y_3$, where 
$$
K_{S_0}+B_{S_0}=(K_{V'}+B_{V'}+\sum _{P\in \Sigma}b_P\pi'^*P)|_{S_0}.
$$
This means that there is a common resolution 
\begin{equation*}
\xymatrix{ & T\ar[dl]_{\alpha} \ar[dr]^{\beta}\\
 W_3 \ar@{-->}[rr]  & & S_0}
\end{equation*} 
over $Y_3$ such that 
$$
\alpha^*(K_{W_3}+B_{W_3}+\sum _Pb_P g_3^*P)=\beta^*(K_{S_0}+B_{S_0}). 
$$
This implies that $b_P=b_P^{\min}$ for every 
$P\in \Sigma$. Therefore, we have $\mathbf B_{Y_3}=\mathbf B_{Y_3}^{\min}$. 
\end{proof}
Then we obtain 
$$\mathbf M_{Y_3}\sim _{\mathbb Q}\mathbf M_{Y_3}^{\min}=\sigma_3^*\mathbf M_{Y_2}^{\min}
$$ 
because 
$$
K_{Y_3}+\mathbf M_{Y_3}+\mathbf B_{Y_3}\sim _{\mathbb Q} K_{Y_3} 
+\mathbf M_{Y_3}^{\min} +\mathbf B_{Y_3}^{\min}. 
$$
Thus, $\mathbf M_{Y_3}$ is nef and abundant. 
Since 
$$
\mathbf M_{Y_3}=\sigma_3^*\mathbf M_{Y_2}=\sigma_3^*\sigma_2^*\mathbf M_{Y_1}, 
$$ 
$\mathbf M$ is b-nef and abundant. 
Moreover, by replacing $Y_3$ with a suitable 
generically finite cover, we have that 
$\mathbf M_{Y_3}$ and $\mathbf M_{Y_3}^{\min}$ are both Cartier 
(see \cite[Lemma 5.2 (5), Proposition 5.4, and Proposition 5.5]{ambro1}) 
and $\mathbf M_{Y_3}\sim \mathbf M_{Y_3}^{\min}$.  
\end{proof}

We close this paper with a remark on the b-semi-ampleness of $\mathbf M$. 
For some related topics, see \cite[Section 3]{fujino10}. 

\begin{rem}[b-semi-ampleness]\label{41}
Let $f:X\to Y$ be a projective surjective morphism between normal projective varieties with 
connected fibers. Assume that 
$(X, B)$ is log canonical and $K_X+B\sim _{\mathbb Q, Y}0$. Without loss of 
generality, we may assume that $(X, B)$ is dlt by taking a dlt blow-up. We set 
$$
d_f(X, B)=\left\{\dim W-\dim Y  \left|
\begin{array}{l}  {\text{$W$ is a log canonical center of}}\\
{\text{$(X, B)$ which is dominant onto $Y$}} 
\end{array}\right. \right\}. 
$$
If $d_f(X, B)\in \{0, 1\}$, then the b-semi-ampleness of the moduli part $\mathbf M$ 
follows from \cite{kawamata} and \cite{ps} by the 
proof of Theorem \ref{main}. Moreover, it is obvious that $\mathbf M\sim _\mathbb Q 0$ when $d_f(X, B)=0$. 
\end{rem} 


\begin{thebibliography}{KKMS}

\bibitem[A1]{ambro1} 
F.~Ambro, Shokurov's boundary property, 
J. Differential Geom. {\textbf{67}} (2004), no. 2, 229--255.

\bibitem[A2]{ambro2} 
F.~Ambro, The moduli b-divisor of an lc-trivial fibration, 
Compos. Math. {\textbf{141}} (2005), no. 2, 385--403. 

\bibitem[BC]{birkar-chen} 
C.~Birkar, Y.~Chen, 
On the moduli part of the Kawamata--Kodaira canonical bundle formula, preprint (2012). 

\bibitem[C]{corti} 
A.~Corti, $3$-fold flips after Shokurov, 
{\em{Flips for $3$-folds and $4$-folds}}, 18--48, 
Oxford Lecture Ser. Math. Appl., {\textbf{35}}, Oxford Univ. Press, Oxford, 2007.

\bibitem[D]{deligne}
P.~Deligne, 
Th\'eorie de Hodge.~II, 
(French) Inst. Hautes \'Etudes Sci. Publ. Math. No. {\textbf{40}} (1971), 5--57.

\bibitem[Fl]{floris} 
E.~Floris, Inductive approach to effective b-semiampleness, to appear 
in Int. Math. Res. Not. IMRN. 

\bibitem[F1]{fujino-abun} 
O.~Fujino, Abundance theorem for semi log canonical threefolds, 
Duke Math. J. {\textbf{102}} (2000), no. 3, 513--532.

\bibitem[F2]{fujino1} 
O.~Fujino, A canonical bundle formula for certain algebraic fiber spaces and its applications, 
Nagoya Math. J. {\textbf{172}} (2003), 129--171.

\bibitem[F3]{fujino-pre} 
O.~Fujino, Higher direct images of log canonical divisors 
and positivity theorems, preprint (2003). arXiv:math/0302073v1 

\bibitem[F4]{fujino2} 
O.~Fujino, 
Higher direct images of log canonical divisors, 
J. Differential Geom. {\textbf{66}} (2004), no. 3, 453--479. 

\bibitem[F5]{fujino-what} 
O.~Fujino, What is log terminal?, 
{\em{Flips for $3$-folds and $4$-folds}}, 49--62, 
Oxford Lecture Ser. Math. Appl., {\textbf{35}}, Oxford Univ. Press, Oxford, 2007.

\bibitem[F6]{fujino-kawa} 
O.~Fujino, 
On Kawamata's theorem, 
{\em{Classification of algebraic varieties}}, 305--315, EMS Ser. Congr. Rep., Eur. Math. Soc., Z\"urich, 2011.

\bibitem[F7]{fujino3} 
O.~Fujino, 
Semi-stable minimal model program for varieties with trivial canonical divisor, 
Proc. Japan Acad. Ser. A Math. Sci. {\textbf{87}} (2011), no. 3, 25--30.

\bibitem[F8]{fujino-funda} 
O.~Fujino, Fundamental theorems for the log minimal model program, 
Publ. Res. Inst. Math. Sci. {\textbf{47}} (2011), no. 3, 727--789.

\bibitem[F9]{fujino4} 
O.~Fujino, 
Basepoint-free 
theorems:~saturation, b-divisors, and canonical bundle formula, 
Algebra Number Theory {\textbf{6}} (2012), no. 4, 797--823. 

\bibitem[F10]{fujino10}
O.~Fujino, Some remarks on the minimal model program for log canonical 
pairs, preprint (2013). 

\bibitem[FF]{fujino-fujisawa} 
O.~Fujino, T.~Fujisawa, 
Variations of mixed Hodge structure and semi-positivity theorems, 
to appear in Publ. Res. Inst. Math. Sci.  

\bibitem[Kat]{katz} 
N.~M.~Katz, 
The regularity theorem in algebraic geometry, 
Actes du Congr\`es International des Math\'ematiciens (Nice, 1970), 
Tome 1, pp. 437--443. Gauthier-Villars, Paris, 1971. 

\bibitem[Kaw]{kawamata} 
Y.~Kawamata, Subadjunction of log canonical divisors for a subvariety of codimension $2$, 
{\em{Birational algebraic geometry (Baltimore, MD, 1996)}}, 79--88, Contemp. Math., {\textbf{207}}, 
Amer. Math. Soc., Providence, RI, 1997.

\bibitem[KKMS]{kkms} 
G.~Kempf, F.~Knudsen, D.~Mumford, B.~Saint-Donat, 
{\em{Toroidal embeddings.~I}}, 
Lecture Notes in Mathematics, Vol. {\textbf{339}}. Springer-Verlag, Berlin-New York, 1973. 

\bibitem[Kol]{kollar} 
J.~Koll\'ar, Kodaira's canonical bundle formula and adjunction, 
{\em{Flips for $3$-folds and $4$-folds}}, 134--162, 
Oxford Lecture Ser. Math. Appl., {\textbf{35}}, Oxford Univ. Press, Oxford, 2007

\bibitem[PS]{ps} 
Yu.~G.~Prokhorov, V.~V.~Shokurov, 
Towards the second main theorem on complements, 
J. Algebraic Geom. {\textbf{18}} (2009), no. 1, 151--199. 

\bibitem[SSU]{ssu} 
M.-H.~Saito, Y.~Shimizu, S.~Usui, 
Variation of mixed Hodge structure and the Torelli problem, 
{\em{Algebraic geometry, Sendai, 1985}}, 649--693, Adv. Stud. Pure Math., {\textbf{10}}, 
North-Holland, Amsterdam, 1987.
\end{thebibliography}
\end{document}